\documentclass[12pt,a4paper]{article}

\usepackage[latin2]{inputenc}
\usepackage{amssymb}
\usepackage{paralist}
\usepackage{amsmath, calc}
\usepackage{amsfonts}
\usepackage{amsthm}
\usepackage{graphicx}
\usepackage{fullpage}
\usepackage{enumerate}
\newtheorem{thm}{Theorem}
\newtheorem{prop}{Proposition}

\newtheorem{lem}{Lemma}

\begin{document}

\title{Generalized asymptotic Sidon basis}
\author{S\'andor Z. Kiss \thanks{Institute of Mathematics, Budapest
University of Technology and Economics, H-1529 B.O. Box, Hungary;
kisspest@cs.elte.hu;
This author was supported by the National Research, Development and Innovation Office NKFIH Grant No. K115288 and K129335. 
This paper was supported by the J\'anos Bolyai Research Scholarship of the Hungarian Academy of Sciences. Supported by the \'UNKP-18-4 New National Excellence Program of the Ministry of 
Human Capacities. Supported by the \'UNKP-19-4 New National Excellence Program 
of the Ministry for Innovation and Technology}, Csaba
S\'andor \thanks{Institute of Mathematics, Budapest University of
Technology and Economics, MTA-BME Lendület Arithmetic Combinatorics Research Group H-1529 B.O. Box, Hungary, csandor@math.bme.hu.
This author was supported by the NKFIH Grants No. K129335. Research supported by the Lendület program of the Hungarian Academy of Sciences (MTA), under grant number LP2019-15/2019.} 
}
\date{}
\maketitle

\begin{abstract}
\noindent Let $h,k \ge 2$ be integers. We
say a set $A$ of positive integers is an asymptotic basis of order
$k$ if every large enough positive integer can be represented as the sum of
$k$ terms from $A$. A set of positive integers $A$ is
called $B_{h}[g]$ set if all positive integers can be represented as the sum
of $h$ terms from $A$ at most $g$ times. In this paper we prove the
existence of $B_{h}[1]$ sets which are asymptotic bases of order $2h+1$
by using probabilistic methods. 


\end{abstract}
\section{Introduction}

Let $\mathbb{N}$ denote the set of positive integers. Let $h, k \ge 2$ be
integers. Let $A \subset \mathbb{N}$ be an infinite set
of positive integers and let $R_{h,A}(n)$ denote
the number of solutions of the equation 
\begin{equation}
a_{1} + a_{2} + \dots + a_{h} = n, \hspace*{3mm} a_{1} \in
A, \dots, a_{h} \in A, \hspace*{3mm} a_{1} \le
a_{2} \le \dots{} \le a_{h},
\end{equation}
\noindent where $n \in \mathbb{N}$. A set of positive integers $A$ is
called $B_h[g]$ set if for every $n \in \mathbb{N}$, the number of 
representations of $n$ as the sum of $h$ terms in the form (1) is at most $g$,
that is $R_{h,A}(n) \le g$. We denote the fact that $A$ is a $B_h[g]$ set by $A \in B_h[g]$.    
We say a set $A \subset \mathbb{N}$ is an asymptotic basis
of order $k$, if $R_{k,A}(n) > 0$ for all large enough positive
integer $n$, i.e., if there exists a positive integer $n_{0}$ such that
$R_{k,A}(n) > 0$ for $n > n_{0}$. In [4] and [5] P. Erd\H{o}s,
A. S\'ark\"ozy and V. T. S\'os asked if there exists a Sidon set (or
$B_2[1]$ set) which is an asymptotic basis of order 3. 
It is easy to see that a Sidon set cannot be an asymptotic basis of
order 2. J. M. Deshouillers and A. Plagne in
[3] constructed a Sidon set which is an asymptotic basis of order at most
7. In [7] it was proved the existence  of Sidon sets which are asymptotic bases of
order 5 by using probabilistic methods. In [1] and [9] this result was 
improved on by proving the existence of a Sidon set which is an asymptotic basis
of order 4. It was also proved [1] that there exists a $B_2[2]$ set which is 
an asymptotic basis of order 3. 
In this paper we will prove a similar but more general theorem. Namely, we prove the existence of an asymptotic basis of order $2h+1$ which is a $B_{h}[1]$ set.
 
\begin{thm}
For every $h \ge 2$ integer there exists a
$B_{h}[1]$ set which is an asymptotic basis of order $2h+1$. 
\end{thm}

Before we prove the above theorem, we give a short survey of the probabilistic
method we are working with.  

\section{Probabilistic tools}

To prove Theorem 1 we use the probabilistic method due to Erd\H{o}s
and R\'enyi. There is an excellent summary of this method in the book of
Halberstam and Roth [6]. In this paper we denote the probability of an event
by $\mathbb{P}$, and the expectation of a random variable $Y$ 
by $\mathbb{E}(Y)$. Let $\Omega$ denote the set of the strictly increasing 
sequences of positive integers. 

\begin{lem}
Let 
\[
\alpha_{1}, \alpha_{2}, \alpha_{3} \dots{} 
\]
be real numbers satisfying 
\[
0 \le \alpha_{n} \le 1 \hspace*{4mm} (n = 1, 2, \dots{}).
\]
\noindent Then there exists a probability space ($\Omega$, $X$, $\mathbb{P}$) with the
following two properties:
\begin{itemize}
\item[(i)] For every natural number $n$, the event $E^{(n)} =
	   \{A$:
  $A \in \Omega$, $n \in A\}$ is measurable, and
  $\mathbb{P}(E^{(n)}) = \alpha_{n}$.
\item[(ii)] The events $E^{(1)}$, $E^{(2)}, \dots{}$
	    are independent. 
\end{itemize}
\end{lem}
See Theorem 13. in [6], p. 142. 
We denote the characteristic function of the event $E^{(n)}$
by $\varrho(A, n)$: 
\[
\varrho(A, n) = 
\left\{
\begin{aligned}
1 \textnormal{, if } n \in A \\
0 \textnormal{, if } n \notin A.
\end{aligned} \hspace*{3mm}
\right.
\]

\noindent Furthermore, for some $A = \{a_1, a_2, \dots{}\} \in
\Omega$ we denote the number of solutions of
$a_{i_{1}} + a_{i_{2}} + \dots{} + a_{i_{h}} = n$ with 
$a_{i_{1}} \in A$, $a_{i_{2}} \in A, \dots{} ,a_{i_{h}}
\in A$, $1 \le a_{i_{1}} < a_{i_{2}} \dots{}
< a_{i_{h}} < n$ by $r_{h}(n)$.  
Then 
\begin{equation}
r_{h,A}(n) = r_{h}(n) = \sum_{\overset{(a_{1}, a_{2},
 \dots{}, a_{h}) \in \mathbb{N}^{h}}{1 \le a_{1} < \dots{} < a_{h} <
 n}\atop {a_{1} + a_{2} + \dots{} + a_{h} =
    n}}\varrho(A, a_{1})\varrho(A, a_{2}) \dots{}
\varrho(A, a_{h}).
\end{equation}
\noindent 
Let $R_{h}^{*}(n)$ denote the number of those representations of $n$
in the form (1) in which there are at least two  equal terms. Thus we have 
\begin{equation}
R_{h,A}(n) = r_{h}(n) + R_{h}^{*}(n).
\end{equation}

\noindent In the proof of Theorem 1 we use the following lemma:

\begin{lem}(Borel-Cantelli)
Let $X_{1}, X_{2}, \dots{}$ be a sequence of events in a probability space. If 
\[
\sum_{j=1}^{+\infty}\mathbb{P}(X_{j}) < \infty,
\]

\noindent then with probability 1, at most a finite number of the events
$X_{j}$ can occur.
\end{lem}
\noindent See [6], p. 135.

\section{Proof of Theorem 1}

Let $h$ be fixed and let $\alpha = \frac{2}{4h+1}$. Define the sequence
$\alpha_{n}$ in Lemma 1 by

\[
\alpha_{n} = \frac{1}{n^{1-\alpha}},
\]

\noindent so that $\mathbb{P}(\{A$:
  $A \in \Omega$, $n \in A\}) =
  \frac{1}{n^{1-\alpha}}$. The proof of Theorem 1 consists of three parts. In
  the first part we prove similarly as in [8] that with probability 1,
  $A$ is an asymptotic basis of order $2h + 1$. In particular, we show that $R_{2h+1,A}(n)$ tends to infinity as $n$ goes to infinity. In the second part we show that deleting finitely many elements from $A$ we obtain a $B_{h}[1]$ set. 
Finally, we show that the above deletion does not destroy the asymptotic basis property.
\\ 
By (3), to prove that $A$ is an asymptotic basis of order $2h + 1$ it is enough to show $r_{2h+1,A}(n) > 0$ for every $n$ large enough. To do this, we 
apply the following lemma with $k = 2h+1$.

\begin{lem}
Let $k \ge 2$ be a fixed integer and let $\mathbb{P}(\{A$:
  $A \in \Omega$, $n \in A\}) =
\frac{1}{n^{1-\alpha}}$ where $\alpha > \frac{1}{k}$. Then with probability 1,
 $r_{k,A}(n) > cn^{k\alpha-1}$ for every sufficiently large $n$,  
where $c = c(\alpha,k)$ is a positive constant. 
\end{lem}
\noindent The proof of Lemma 3 can be found in [8]. It is clear from (3) that 
\begin{equation}
\mathbb{P}(\mathcal{E}) = 1,
\end{equation}
where $\mathcal{E}$ denotes the event
\[
\mathcal{E} = \{A: A \in \Omega, \exists n_{0}(A) = n_{0} \hspace*{1mm} such \hspace*{1mm} that  \hspace*{1mm}
			 R_{2h+1,A}(n) \ge cn^{\frac{1}{4h+1}}  \hspace*{1mm} for \hspace*{1mm} n>n_{0}\},
\]
where $c$ is a suitable positive constant.
In the next step we prove that removing finitely elements from $A$ we get a $B_{h}[1]$ set with probability 1. To do this, it is enough to show that with probability 1, $R_{h,A}(n) \le 1$ for every $n$ large enough.
Note that in a representation of $n$ as the sum of $h$ terms there can be equal summands. To handle this situation we consider the terms of a representation $a_{1} + \dots{} + a_{h} = n$
as a vector $(a_{1}, \dots{}, a_{h}) \in \mathbb{N}^{h}$. 
We denote the set which elements are the coordinates of the vector
$\bar{x}$ as $Set(\bar{x})$. Of course, if two or more coordinates of $\bar{x}$ are equal, this value appears only once in
$Set(\bar{x})$.
We say that two vectors $\bar{x}$ and $\bar{y}$ are disjoint if
$Set(\bar{x})$ and $Set(\bar{y})$ are disjoint sets. We define
$r^*_{l, A} (n)$ as the maximum number of pairwise disjoint
representations of $n$ as sum of $l$ elements of $A$, i.e., the
maximum number of pairwise disjoint vectors of $R_{l}(n)$ with
their coordinates in $A$. We say that $A$ is a $B^*_l[g]$ sequence
if $r^*_{l,A}(n)\le g$ for every $n$.

\begin{lem}
Let $\mathbb{P}(\{A$:$A \in \Omega$, $n \in A\}) = \frac{1}{n^{1-\alpha}}$, where
$\alpha = \frac{2}{4h+1}$. 
\begin{itemize}
\item[(i)] For every $2 \le k \le h$ almost always there exists a finite set $A_{k}$ such that $r^*_{k, A\setminus A_{k}}(n) \le 1$.
\item[(ii)] For every $h+1 \le k \le 2h$ almost always there exists a finite set $A_{k}$ such that $r^*_{k, A\setminus A_{k}}(n) \le 4h+1$.
\end{itemize}
\end{lem}

\begin{proof}
We need the following proposition (see Lemma 3.7 in [2]).

\begin{prop} For a sequence $A \in \Omega$, for every $k$ and $n$
$$\mathbb{P}(r^*_{k, A}(n) \ge s) \le C_{k,\alpha, s} \  n^{(k\alpha - 1)s}$$
where $C_{k, \alpha, s}$ depends only on $k$, $\alpha$ and $s$.
\end{prop}

\noindent We apply Proposition 1 by $s = 2$. Then we have
$$\mathbb{P}(r^*_{k, A}(n) \ge 2) \le C_{k,\alpha} \  n^{2(k\alpha - 1)} = C_{k,\alpha} \  
n^{-\frac{8h-4k+2}{4h+1}}.$$
Since $2 \le k \le h$, we have 
$$\mathbb{P}(r^*_{k, A}(n) \ge 2) \le n^{-\frac{4h+2}{4h+1}}$$ then by the Borel-Cantelli lemma we get that almost always there exists an $n_{k}$ such that 
$r^*_{k, A}(n) \le 1$ for $n \ge n_{k}$. It follows that 

\[
r^*_{k, A\setminus A_{k}}(n) \le 1, 
\]
where $A_{k} = A \cap [0, n_{k}]$.

Assume that $h < k \le 2h$.
We apply Proposition 1 by $s = 4h+2$.  Then we have
$$\mathbb{P}(r^*_{k, A}(n) \ge 4h+2) \le C_{k,h,\alpha} \  n^{(4h+2)(k\alpha - 1)} 
= C_{k,\alpha} \  n^{-(2h+1)\frac{8h-4k+2}{4h+1}}.$$
Since $h < k \le 2h$, we have 
$$\mathbb{P}(r^*_{k, A}(n) \ge 4h+2) \le n^{-\frac{4h+2}{4h+1}}$$ then by the Borel-Cantelli lemma we get that almost always there exists an $n_{k}$ such that 
$r^*_{k, A}(n) \le 4h+1$ for $n \ge n_{k}$. It follows that

\[
r^*_{k, A\setminus A_{k}}(n) \le 4h+1, 
\]
where $A_{k} = A \cap [0, n_{k}]$.
\end{proof}
It follows from (4) and Lemma 4 that there exists a set $A$ and for every $2 \le k \le h$ finite sets $A_{k} \subset A$ such that
\begin{equation}
R_{2h+1,A}(n) \ge cn^{\frac{1}{4h+1}}
\end{equation}
for $n \ge n_{0}$ and for every $2 \le k \le h$,
\begin{equation}
r^*_{k, A\setminus A_{k}}(n) \le 1,
\end{equation}
for  every $h < k \le 2h$,
\begin{equation}
r^*_{k, A}(n) \le 4h+1.
\end{equation}
Set $B = A \setminus \cup_{k=1}^{2h}A_{k}$.
In the next step we show that $B$ is both a $B_{h}[1]$ set and a $B_{2h}[g]$ set for 
some $g$. We apply the following proposition (see Remark 3.10 in [2]).

\begin{prop} 
$$B^*_h[g] \cap B_{h-1}[l] \subseteq B_h[g(h(l-1)+1)].$$
\end{prop}

\noindent By using the definition of $B$, the fact that $B_{2}^{*}[1] = B_{2}[1]$ and (6), (7) it follows that 
\[
B \in B_{2}[1] \cap B_{3}^{*}[1] \cap \dots{} \cap B_{h}^{*}[1] 
\cap B_{h+1}^{*}[4h+1] \cap \dots{} \cap B_{2h}^{*}[4h+1].
\]
Applying Proposition 2 with $g = l = 1$ we get by induction that for every $2 \le s \le h$
if $B \in B^*_{s}[1] \cap B_{s-1}[1]$ then $B \in B_{s}[1]$, thus $B$ is a $B_{h}[1]$ set.
Applying Proposition 2 with $g = 4h + 1$, $l = 1$ we get that $B \in B_{h+1}[4h+1]$. Using Proposition 2
again with $g = 4h + 1$, $l = 4h+1$ we get that if $B \in B^*_{h+2}[4h+1] \cap B_{h+1}[4h+1]$ then $B \in B_{h+2}[(4h+1)(h\cdot 4h + 1)]$.
Continuing this process we obtain that for every $1 < k \le 2h$ we have $B \in B_{k}[g_{k}]$ for some positive integer $g_{k}$.
Let $\cup_{k=1}^{2h}A_{k} = \{d_{1}, \dots{} ,d_{w}\}$ $(d_{1} < \dots{} < d_{w})$. Now we show that $A \in B_{2h}[G]$ where 
\[
G = 2^{w}\cdot \max_{1 < k \le 2h}g_{k}. 
\]
We prove by contradiction.
Assume that there exists a positive integer $n$ with $R_{2h,A}(n) > 2^{w}\cdot \max_{1 < k \le 2h}g_{k}$. Then there exist indices $1 \le i_{1} < i_{2} < \dots{} < i_{j} \le w$
such that the number of representations of in the form $n = d_{i_{1}} + \dots{} + d_{i_{j}} + c_{j+1} + \dots{} + c_{2h}$, where $c_{j+1}, \dots{} ,c_{2h} \in B$ is more than $\max_{1 < k \le 2h}g_{k}$.
It follows that
\[
R_{2h-j,B}(n-(d_{i_{1}} + \dots{} + d_{i_{j}})) > \max_{1 < k \le 2h}g_{k} \ge g_{2h-j}
\]   
which is a contradiction.

Finally, we prove similarly as in [7] that
$B$ is an asymptotic basis of order $2h + 1$, i.e., the deletion of the
``small''elements of $A$ does not destroy its asymptotic basis
property. We prove by contradiction. Assume that
there exist infinitely many positive integers $M$ which cannot be represented as
the sum of $2h + 1$ numbers from $B$. Choose such an $M$ large enough. In view of (5), we have $R_{2h+1,A}(M) > cM^{\frac{1}{4h+1}}$. It
follows from our assumption that every representations of $M$ as the sum of $2h + 1$
numbers from $A$ contains at least one element from
$A \setminus B = \cup_{k=1}^{2h}A_{k}$.
Then by the pigeon hole principle there exists an $y \in \cup_{k=1}^{2h}A_{k}$ which is in at least $\frac{R_{2h+1,A}(M)}{w}$ representations of $M$. As $A \in  B_{2h}[G]$, it follows that with probability 1,
\[
\frac{c_{3}M^{\frac{1}{4h+1}}}{w} < \frac{R_{2h+1,A}(M)}{w} \le R_{2h,A}(M-y) \le G,
\]
\noindent which is a contradiction if $M$ is large enough.


\end{document}